\documentclass[11pt]{article}

\usepackage{amssymb,amsmath,amsthm}
\usepackage{geometry}
\usepackage{color}
\usepackage{varwidth}
\usepackage{subfigure}
\usepackage{graphicx}
\usepackage{tikz}
\usetikzlibrary{arrows,positioning} 
\usetikzlibrary{shapes.geometric}

\newcommand{\e}{{\bf e}}
\newcommand{\ux}{{\underline{x}}}

\newcommand{\p}{{\partial}}

\newcommand{\gso}{\mathfrak{so}}

\newtheorem{thm}{Theorem}
\newtheorem{lemma}{Lemma}
\newtheorem{example}[thm]{Example}
\newtheorem{rem}{Remark}

\newtheorem{defi}{Definition}
\newtheorem{cor}{Corollary}

\title{Spinorspaces in discrete Clifford analysis}
\date{}
\author{H. De Ridder\footnote{Ghent University, Department of Mathematical Analysis, Building S22, Galglaan 2, 9000 Gent, Belgium, email: hdr@cage.ugent.be}, \ T. Raeymaekers\footnote{Ghent University, Department of Mathematical Analysis, Building S22, Galglaan 2, 9000 Gent, Belgium, email: tr@cage.ugent.be}}

\begin{document}

\maketitle

\begin{abstract}
In this paper we work in the `split' discrete Clifford analysis setting, i.e. the $m$-dimensional function theory concerning null-functions, defined on the grid $\mathbb{Z}^m$, of the discrete Dirac operator $\p$, involving both forward and backward differences, which factorizes the (discrete) Star-Laplacian ($\Delta^\ast = \p^2$). We show how the space $\mathcal{M}_k$ of discrete homogeneous spherical monogenics of degree $k$, is decomposable into $2^{2m-n}$ isomorphic irreducible representations with highest weight $\left(k + \frac{1}{2}, \frac{1}{2},\ldots,\frac{1}{2}\right)$ in the odd-dimensional case and two times $2^{2m-n}$ isomorphic irreducible representations with highest weight $(k)'_+ = \left(k + \frac{1}{2}, \frac{1}{2},\ldots,\frac{1}{2},\frac{1}{2}\right)$ resp. $(k)'_- = \left(k + \frac{1}{2}, \frac{1}{2},\ldots,\frac{1}{2},-\frac{1}{2}\right)$ in the even dimensional case. 



\end{abstract}

\section{Introduction}
In classical Clifford analysis, the infinitesimal `rotations' are given by the angular momentum operators $L_{a,b} = x_a \partial_{x_b} - x_b \partial_{x_a}$. These operators satisfy the commutation relations
$$
[L_{a,b}, L_{c,d}] =  \delta_{b,c} \, L_{a,d} - \delta_{b,d}\, L_{a, c} - \delta_{a, c} \, L_{b,d} + \delta_{a,d} \, L_{b, c},
$$
which are exactly the defining relations of the special orthogonal Lie algebra $\gso(m)$ and they form endomorphisms of the space $\mathcal{H}_k(m,\mathbb{C})$ of scalar-valued harmonic homogeneous polynomials, thus transforming the latter in an (irreducible) $\gso(m,\mathbb{C})$-representation. To establish $\mathcal{M}_k(m,\mathbb{S})$, i.e. the spinor-valued homogeneous monogenics of degree $k$, classically as $\gso(m,\mathbb{C})$-representation, the following operators are considered
$$
dR(e_{a,b}): \mathcal{M}_k(m,\mathbb{S}) \to \mathcal{M}_k(m,\mathbb{S}), \qquad M_k \mapsto \left( L_{a,b} + \frac{1}{2}Ê\, e_a\,e_b \right) M_k. 
$$
These operators are endomorphisms of the space of spinor-valued $k$-homogeneous polynomials which also satisfy the defining relations of $\gso(m,\mathbb{C})$:
$$
\left[ dR(e_{a,b}), dR(e_{c,d}) \right] =  \delta_{b,c} \, dR(e_{a,d}) - \delta_{b,d}\, dR(e_{a, c}) - \delta_{a, c} \, dR(e_{b,d}) + \delta_{a,d} \, dR(e_{b, c}).
$$

\medskip In \cite{Rotations}, we developed similar operators in the discrete Clifford analysis setting: the angular momentum operators are discrete operators $L_{a,b} = \xi_a \, \p_b + \xi_b \, \p_a$, $a \neq b$. For $a = b$, we define $L_{aa} = 0$. Then the operators $\Omega_{a,b}$, acting on discrete functions $f$ as $\Omega_{a,b} \, f = L_{a,b}\, f \, e_b \, e_a$, satisfy the defining relations of the special lie algebra $\mathfrak{so}(m)$:
$$
\left[ \Omega_{a,b}, \Omega_{c,d} \right] = \delta_{b,c} \, \Omega_{a,d} - \delta_{b,d}\, \Omega_{a,c} - \delta_{a,c} \, \Omega_{b,d} + \delta_{a,d} \, \Omega_{b,c}.
$$
Furthermore, they are endomorphisms of the space $\mathcal{H}_k$ of Clifford-algebra valued homogeneous harmonics of degree $k$, since $\Omega_{a,b}$ commutes with $\frak{sl}_2 = \left\{ \Delta, \xi^2, \mathbb{E} + \frac{m}{2} \right\}$, $\forall \, (a, b)$. In \cite{paperIsaac}, we showed that $\mathcal{H}_k$ is the sum of $2^{2m}$ isomorphic copies of the irreducible representation of $\gso(m,\mathbb{C})$ with highest weight $(k,0,\ldots,0)$.

\medskip The discrete Dirac operator $\p$ is however not invariant under the operators $\Omega_{a,b}$, hence $\mathcal{M}_k$ cannot be expressed as $\gso(m,\mathbb{C})$-representation by means of these operators. Therefore, we considered in \cite{Rotations} the operators $L_{a,b} - \dfrac{1}{2}$ and the four-vector $V_{a,b} = e_a\, e_b \,e_a^\perp \, e_b^\perp = - e_a^\perp \,e_a\,e_b^\perp \, e_b$. Let the operator $dR(e_{a,b})$, $a \neq b$, act on discrete functions $f$ as 
$$
dR(e_{a,b})\, f = V_{a,b} \left( L_{a,b} - \frac{1}{2}\right) f \, e_a^\perp \, e_b^\perp.
$$
For $a = b$, we defined $dR(e_{a,a}) = 0$. The operators $dR(e_{a,b})$ satisfy the defining relations of the special lie algebra $\mathfrak{so}(m)$:
$$
\left[ dR(e_{a,b}), dR(e_{c,d}) \right] = \delta_{b,c} \, dR(e_{a,d}) - \delta_{b,d}\, dR(e_{a,c}) - \delta_{a,c} \, dR(e_{b,d}) + \delta_{a,d} \, dR(e_{b,c}),
$$
and commute with $\frak{osp}(1 | 2) = \left\{ \p, \xi, \mathbb{E}Ê+ \frac{m}{2}Ê\right\}$ which makes them endomorphisms of the space of $k$-homogeneous discrete monogenic polynomials. As such, the space $\mathcal{M}_k$ of $k$-homogeneous Clifford-valued monogenic polynomials is a reducible $\gso(m,\mathbb{C})$-representation. In \cite{Rotations}, it was already suggested that $\mathcal{M}_k$ can be decomposed into irreducible parts of highest weight $\left(k + \frac{1}{2}, \frac{1}{2},\ldots,\frac{1}{2}\right)$ resp. $\left(k + \frac{1}{2}, \frac{1}{2},\ldots,-\frac{1}{2}\right)$, but this was left as open conjecture. In the following sections, we will show how this decomposition is done exactly.

\section{Preliminaries}
Let $\mathbb{R}^m$ be the $m$-dimensional Euclidian space with orthonormal basis $e_j$, $j=1,\dots,m$ and consider the Clifford algebra $\mathbb{R}_{m,0}$ over $\mathbb{R}^m$. Passing to the so-called `split' discrete setting \cite{frame,SW}, we imbed the Clifford algebra $\mathbb{R}_{m,0}$ into the bigger complex one $\mathbb{C}_{2m,0}$, the underlying vector space of which has twice the dimension, and introduce forward and backward basis elements $ \textbf{e}_j^{\pm}$ satisfying the following anti-commutator rules:
$$
 \left\{ \e_j^-, \e_\ell^-\right\} = \left\{ \e_j^+, \e_\ell^+ \right\} = 0, \qquad 
 \left\{ \e_j^+, \e_\ell^- \right\} = \delta_{j \ell}, \qquad j,\,\ell= 1,\dots,m.
$$
The connection to the original basis $\e_j$ is given by $\e^+_j + \e^-_j = e_j$, $j=1,\ldots,m$. This implies $e_j^2=1$, in contrast to the usual Clifford setting where traditionally $e_j^2=-1$ is chosen. We will often denote $\e_j^+ \wedge \e_j^- = \e_j^+ \e_j^- - \e_j^- \e_j^+$, $j= 1,\ldots,m$. 

\medskip Now consider the standard equidistant lattice $\mathbb{Z}^m$; the coordinates of a Clifford vector $\ux$ will thus only take integer values. We construct a discrete Dirac operator factorizing the discrete Laplacian, using both forward and backward differences $\Delta_j^{\pm}$, $j=1,\dots, m$, acting on Clifford-valued functions $f$ as follows:
$$
\Delta^+_j[f](\cdot) = f(\cdot +\e_j) - f(\cdot), \qquad \Delta^-_j[f](\cdot) = f(\cdot) - f(\cdot-\e_j). 
$$
With respect to the $\mathbb{Z}^m$-grid, the usual definition of the discrete Laplacian in $\ux \in \mathbb{Z}^m$ is
$$
\Delta^{\ast}[f](\ux) = \sum_{j=1}^m \Delta^+_j \Delta^-_j[f] = \sum_{j=1}^m \left( f(\ux+\e_j) + f(\ux - \e_j) \right) - 2 m\, f(\ux).
$$
This operator is also known as ``Star Laplacian"; we will from now on simply write $\Delta$. An appropriate definition of a discrete Dirac operator $\p$ factorizing $\Delta$, i.e. satisfying $\p^2 = \Delta$, is obtained by combining the forward and backward basis elements with the corresponding forward and backward differences, more precisely  
$$
\p = \sum_{j=1}^m \left(\e^+_j \Delta^+_j + \e^-_j \Delta^-_j \right). 
$$
In order to receive an analogue of the classical Weyl relations $\p_{x_j} x_k - x_k \p_{x_j} = \delta_{jk}$, the co-ordinate vector variable operators $\xi_j = \e_j^+ \, X_j^- + \e_j^- \,X_j^+$ are defined by their interaction with the corresponding co-ordinate operators $\p_j = \e_j^+ \, \Delta_j^+ + \e_j^- \,\Delta_j^-$, according to the skew Weyl relations, cf. \cite{SW}
$$
\p_j \, \xi_j - \xi_j \, \p_j = 1, \ j= 1,\ldots,m,
$$
which imply that $\p_j\, \xi_j^k[1] = k \,\xi_j^{k-1}[1]$. The operators $\xi_j$ and $\p_j$ furthermore satisfy the following anti-commutator relations:
$$
\left\{ \xi_j, \xi_k \right\} = \left\{Ê\p_j, \p_k \right\} = \left\{ \p_j, \xi_k \right\} = 0, \qquad j \neq k, \; j, k = 1,\ldots,m
$$
implying that $\p_\ell \,\xi_j^k[1] = 0$, $j \neq \ell$.

\medskip The natural powers $\xi^k_j[1]$ of the operator $\xi_j$ acting on the ground state 1 are the basic discrete $k$-homogeneous polynomials of degree $k$ in the variable $x_j$, i.e. $\mathbb{E}\, \xi_j^k[1] = k\,\xi_j^k[1]$, where $\mathbb{E} = \sum_{j=1}^m \xi_j\,\p_j$ is the discrete Euler operator. They constitute a basis for all discrete polynomials.  Explicit formulas for $\xi_j^k[1]$ are given for example in \cite{SW,CKE}; furthermore $\xi_j^k[1](x_j)=0$ if $ k \geqslant 2\,|x_j|+1$.

\medskip A discrete function is discrete harmonic (resp. left discrete monogenic) in a domain $\Omega \subset \mathbb{Z}^m$ if $\Delta f(\ux) = 0$ (resp. $\p f(\ux) = 0$), for all $\ux \in \Omega$. The space of discrete harmonic (resp. monogenic) homogeneous polynomials of degree $k$ is denoted $\mathcal{H}_k$ (resp. $\mathcal{M}_k$), while the space of all discrete harmonic (resp. monogenic) homogeneous polynomials is denoted $\mathcal{H}$ (resp. $\mathcal{M}$). It is clear that 
$$
\mathcal{H} = \bigoplus_{k=0}^\infty \mathcal{H}_k, \qquad \mathcal{M} = \bigoplus_{k=0}^\infty \mathcal{M}_k.
$$
The respective dimensions over the discrete Clifford algebra are 
$$
\dim(\mathcal{H}_k) = \binom{k+m-1}{k} - \binom{k+m-3}{k}, \qquad \dim(\mathcal{M}_k) = \binom{k+m-2}{k}. 
$$

\section{Orthogonal Lie algebras}
We will start by briefly introducing the orthogonal Lie algebra $\gso(m,\mathbb{C})$; a detailed description can be found for example in \cite{FH}. The orthogonal Lie algebra $\gso(m,\mathbb{C})$ is generated in even dimension $m = 2n$ by $\binom{m}{2}$ basis elements $H_a$, $X_{a,b}$, $Y_{a,b}$ and $Z_{a,b}$ ($1 \leqslant a, b \leqslant n$) and in odd dimension $m=2n+1$ these basis elements are extended to a full basis of $\gso(m,\mathbb{C})$ by $2n$ extra elements $U_a$ and $V_a$, $1 \leqslant a \leqslant n$:
\begin{align*}
\gso(2n,\mathbb{C}) &= \textup{span}_\mathbb{C} \left\{ H_a, X_{a,b}, Y_{a,b}, Z_{a,b}, 1\leqslant a, b \leqslant n, \, a \neq b \right\}, \\ 
\gso(2n+1,\mathbb{C}) &= \textup{span}_\mathbb{C} \left\{ H_a, X_{a,b}, Y_{a,b}, Z_{a,b}, U_a, V_a, 1\leqslant a, b \leqslant n, \, a \neq b \right\}. 
\end{align*}
The Cartan subalgebra can be chosen as
$$
\frak{h}Ê= \left\{ H_a, 1 \leqslant a \leqslant n \right\},
$$
independently of the parity of the dimension, i.e. $\gso(2n,\mathbb{C})$ and $\gso(2n+1,\mathbb{C})$ are both Lie algebras of rank $n$. The roots of $\gso(m,\mathbb{C})$ (see also \cite{Knapp}) are determined by considering the adjoint representation ($1 \leqslant a, b, c, d \leqslant n$):
\begin{align*}
\left[ H_c, Y_{a,b}\right] &= \left( \delta_{ca} + \delta_{cb}\right) Y_{a,b} = \left(\left( L_a + L_b \right)(H_c)\right) Y_{a,b}, \\
\left[ H_c, X_{a,b}\right] &= \left( \delta_{ca} - \delta_{cb}\right) X_{a,b} = \left(\left( L_a - L_b \right)(H_c)\right) X_{a,b}, \\
\left[ H_c, Z_{a,b}\right] &= -\left( \delta_{ca} + \delta_{cb}\right) Z_{a,b} = \left(\left( - L_a - L_b \right)(H_c)\right) Z_{a,b}, \\
\left[ H_c, U_{a}\right] &= \delta_{ca}\, U_{a} = \left( L_a(H_c)\right) U_{a}, \\
\left[ H_c, V_{a}\right] &= -\delta_{ca}\, U_{a} = \left( -L_a(H_c)\right) U_{a}.
\end{align*}
Note in particular that the Cartan subalgebra elements $H_a$ can be found by means of the commutator of a positive root with a negative root of the same index:
$$
\left[ Y_{a,b}, Z_{a,b}\right] = -H_a - H_b, \qquad \left[ X_{a,b}, X_{b,a} \right]Ê= H_a - H_b. 
$$
We thus deduce the following roots and root vectors. Here $\left\{ L_a, 1 \leqslant a \leqslant n \right\}$ is a basis of the dual vector space $\frak{h}^\ast$ of the Cartan subalgebra $\frak{h}$, i.e. $L_a \left(H_b \right) = \delta_{a,b}$. 

{\centering 
\begin{table}[h]
\begin{center}
\begin{tabular}{ccc}
$m = 2n$ && $m = 2n+1$ \\[1ex] 
\begin{tabular}{cc}
root & root vector \\
\hline
$L_a - L_b$ & $X_{a,b}$ \\
$L_a + L_b$ & $Y_{a,b}$ \\
$-L_a - L_b$ & $Z_{a,b}$ \\
& \\
&
\end{tabular} 
& \hspace{2cm} &
\begin{tabular}{cc}
root & root vector \\
\hline
$L_a - L_b$ & $X_{a,b}$ \\
$L_a + L_b$ & $Y_{a,b}$ \\
$-L_a - L_b$ & $Z_{a,b}$ \\
$L_a$ & $U_a$ \\
$-L_a$ & $V_a$
\end{tabular}
\end{tabular}
\end{center}
\end{table}
}

By the usual convention, we choose the positive roots in even dimension to be 
$$
\left\{ L_a + L_b : 1 \leqslant a \neq b \leqslant n \right\} \cup \left\{ L_a - L_b : 1 \leqslant a < b \leqslant n \right)
$$
and negative roots
$$
\left\{ -L_a - L_b : 1 \leqslant a \neq b \leqslant n \right\} \cup \left\{ L_a - L_b : 1 \leqslant b < a \leqslant n \right).
$$
In odd dimension, one finds positive roots 
$$
\left\{ L_a + L_b : 1 \leqslant a \neq b \leqslant n \right\} \cup \left\{ L_a - L_b : 1 \leqslant a < b \leqslant n \right) \cup \left\{ L_a: 1 \leqslant a \leqslant n \right\}
$$
and negative roots
$$
\left\{ -L_a - L_b : 1 \leqslant a \neq b \leqslant n \right\} \cup \left\{ L_a - L_b : 1 \leqslant b < a \leqslant n \right) \cup \left\{ -L_a: 1 \leqslant a \leqslant n \right\}.
$$

\medskip In \cite{Rotations}, we introduced the algebra $\gso(m,\mathbb{C})$ (up to an isomorphism) in the discrete Clifford analysis context. The generators of $\gso(m,\mathbb{C})$ were not given in terms of the root vectors and Cartan subalgebra, but rather by the generators $\left\{ dR(e_{a,b}): 1 \leqslant a \neq b \leqslant m \right\}$, satisfying the defining relations of $\gso(m,\mathbb{C})$:
\begin{equation} \label{eq:commutator_rule_dR}
\left[ dR(e_{a,b}), dR(e_{c,d})Ê\right] = \delta_{a,d} \, dR(e_{b,c}) + \delta_{b,c}Ê\, dR(e_{a,d}) - \delta_{a,c} \, dR(e_{b,d}) - \delta_{b,d}Ê\, dR(e_{a,c}).
\end{equation} 
In the following sections, we will re-establish the orthogonal Lie algebra in the discrete Clifford analysis setting, but now by determining the explicit expressions of the root vectors and Cartan subalgebra.

\section{Decomposition of $\mathcal{M}_k$ in irreducible representations}
\subsection{Even dimension $m =2n$}
\begin{defi}
We define the operators $H_a$, $X_{a,b}$, $Y_{a,b}$ and $Z_{a,b}$ $\in \gso(m,\mathbb{C})$:
\begin{align*}
H_a &= i \, dR(e_{2a-1,2a}), \qquad 1 \leqslant a \leqslant n, \\
X_{a,b} &= \frac{1}{2} \left( dR(e_{2a-1,2b-1}) + i \, dR(e_{2a-1,2b}) - i \, dR(e_{2a,2b-1}) + dR(e_{2a,2b})Ê\right), \\
Y_{a,b} &= \frac{1}{2} \left( dR(e_{2a-1,2b-1}) - i \, dR(e_{2a-1,2b}) - i \, dR(e_{2a,2b-1}) - dR(e_{2a,2b})Ê\right), \\
Z_{a,b} &= \frac{1}{2} \left( dR(e_{2a-1,2b-1}) + i \, dR(e_{2a-1,2b}) + i \, dR(e_{2a,2b-1}) - dR(e_{2a,2b})Ê\right),
\qquad 1 \leqslant a, b \leqslant n.
\end{align*}
Note that, because $dR(e_{a,b}) = - dR(e_{b,a})$, we find that $Y_{b,a}Ê= -Y_{a,b}$ and $Z_{b,a} = -Z_{a,b}$. For $X_{a,b}$, we find that $X_{b,a} \neq X_{a,b}$ and that $X_{a,a} = H_a$, hence we will only consider couples $(a,b)$ with $a \neq b$. 
\end{defi}

We will now show that these operators indeed show the expected commutator relations:
\begin{lemma}
The operators $H_c$, $X_{a,b}$, $Y_{a,b}$ and $Z_{a,b}$, $1 \leqslant a, b, c, d \leqslant n$, satisfy the commutator relations given in Lemma \ref{lem:commrelso}; in particular:
\begin{align*}
\left[ H_c, Y_{a,b}\right] &= \left( \delta_{ca} + \delta_{cb}\right) Y_{a,b} = \left( L_a + L_b\right)(H_c) \, Y_{a,b}, \\
\left[ H_c, X_{a,b}\right] &= \left( \delta_{ca} - \delta_{cb}\right) X_{a,b} = \left( L_a - L_b\right)(H_c) \, X_{a,b}, \\
\left[ H_c, Z_{a,b}\right] &= -\left( \delta_{ca} + \delta_{cb}\right) Z_{a,b} = -\left( L_a + L_b\right)(H_c) \, Z_{a,b}, \\
\left[ X_{a,b}, Y_{c,d}Ê\right] &= \delta_{bc}Ê\, Y_{a,d} - \delta_{bd}Ê\, Y_{a,c}. 
\end{align*}
In particular, $X_{a,b}$, $a < b$ resp. $Y_{a,b}$ are root vectors corresponding to the positive roots $L_a - L_b$, resp. $L_a + L_b$. Furthermore, $X_{a,b}$ with $a > b$ and $Z_{a,b}$ are root vectors corresponding to the negative roots $L_a - L_b$ resp. $-L_a - L_b$. 
\end{lemma}
\begin{proof}
Since the commutator relations between the operators $dR(e_{a,b})$ are the same as those between the operators $\Omega_{a,b}$ of the harmonics, the proof is completely similar as the proof in \cite{paperIsaac}. 
\end{proof}

\medskip We already established in \cite{rotations} that although $\mathcal{M}_{k}$ is a representation of $\gso(2n,\mathbb{C})$, by means of the operators $dR(e_{a,b})$, acting on $\mathcal{M}_k$, this representation is not irreducible. The decomposition is done by splitting $1$ into a sum of idempotents. We will now introduce the appropriate idempotents for this situation. For a function $P_k \, L$ to be an eigenfunction of the maximal abelian subgroup $\frak{h}$, it must certainly hold that $L \, e_{2a-1}^\perp \, e_{2a}^\perp$ is again equal to $L$ up to a (complex) constant. Consider, for $a = 1, \ldots, n$, the Clifford elements
\begin{align*}
L^\pm_{2a-1} &= \left( \e_{2a-1}^+ \e_{2a-1}^- \pm i\, \e_{2a-1}^+Ê\right), \qquad L^\pm_{2a}Ê= \left( \e_{2a}^+ \e_{2a}^- \pm \e_{2a}^+Ê\right), \\
M^\pm_{2a-1} &= \left( \e_{2a-1}^- \e_{2a-1}^+ \pm i\, \e_{2a-1}^-Ê\right), \qquad M^\pm_{2a}Ê= \left( \e_{2a}^- \e_{2a}^+ \pm \e_{2a}^-Ê\right),
\end{align*}

\medskip For the rest of this article, we will need the following notations. For a factor $F_a \in \left\{ÊL_a^\pm, M_a^\pm \right\}$, $a = 1,\ldots,m$, denote 
$$
|F_a| = \begin{cases}
0, & F_a = L_a^+ \text{ or } M_a^-, \\
1, & F_a = L_a^- \text{ or }ÊM_a^+. 
\end{cases}Ê
\text{ and }Ê
\| F_a \| = \begin{cases}
0, & F_a = L_a^\pm, \\
1, & F_a = M_a^\pm. 
\end{cases}Ê
$$ 
Furthermore, denote by $\widetilde{F}_a$ the idempotent
$$
\widetilde{F}_s = \begin{cases}
L_a^\mp, & \text{Êif }ÊF_a = L_a^\pm, \\
M_a^\mp, & \text{Êif }ÊF_a = M_a^\pm.
\end{cases}
$$
Then $|\widetilde{F}_s|Ê= 1 - |F_s|$ and $\|Ê\widetilde{F}_s \|Ê= \| F_s \|$. 

\begin{lemma} \label{lem:rightmultiplicationL}
The multiplication from the right on the idempotent $F_a \in \left\{ L_a^\pm, M_a^\pm \right\}$ by $e_a$ is given by 
\begin{align*}
F_{2a-1}\,e^\perp_{2a-1} &= (-1)^{|F_{2a}| + 1} \,i\, F_{2a-1}, \\
F_{2a}\,e^\perp_{2a} &= (-1)^{|F_{2a}|+1} \, \widetilde{F}_{2a}. 
\end{align*}
As a result, for $1 \leqslant a \leqslant n$, we have that
\begin{align*}
F_{2a-1} \, F_{2a} \, e_{2a-1}^\perp \, e_{2a}^\perp &= (-1)^{|F_{2a-1}| + |F_{2a}| + 1}Ê\, i \, F_{2a-1}\, F_{2a}.
\end{align*}

\medskip We also find that for $1 \leqslant a < b \leqslant n$ and a general idempotent $F = \prod_{s= 1}^m F_s$, with $F_s \in \left\{ L_s^\pm, M_s^\pm \right\}$, we get
\begin{align*}
V_{2a-1,2b-1} \, F \, e_{2a-1}^\perp \, e_{2b-1}^\perp &= (-1)^{|F_{2a-1}| + |F_{2b-1}| + \|F_{2a-1} \|Ê+ \|F_{2b-1} \| + 1} \, F^{2a,2b-1}, \\
V_{2a-1,2b} \, F \, e_{2a-1}^\perp \, e_{2b}^\perp &= (-1)^{|F_{2a-1}| + |F_{2b}| + \|F_{2a-1} \|Ê+ \|F_{2b} \|} \, i\, F^{2a,2b-1}, \\
V_{2a,2b-1} \, F \, e_{2a}^\perp \, e_{2b-1}^\perp &= (-1)^{|F_{2a}| + |F_{2b-1}| + \|F_{2a} \|Ê+ \|F_{2b-1} \|} \, i \, F^{2a,2b-1}, \\
V_{2a,2b} \, F \, e_{2a}^\perp \, e_{2b}^\perp &= (-1)^{|F_{2a}| + |F_{2b}| + \|F_{2a} \|Ê+ \|F_{2b} \|} \, F^{2a,2b-1}.
\end{align*}
where we denote, for $1 \leqslant s_1 < s_2 \leqslant m$:
$$
F^{s_1, s_2} = F_1 \, F_2 \ldots \, F_{s_1 - 1} \, \widetilde{F}_{s_1} \, \widetilde{F}_{s_1+1} \ldots \, \widetilde{F}_{s_2-1}\, \widetilde{F}_{s_2} \, F_{s_2 + 1} \, F_{s_2 + 2} \ldots F_{m-1} \, F_m.
$$
\end{lemma}

\begin{proof}
Note that 
{\small \begin{align*}
L^\pm_{2a-1} \,e_{2a-1}^\perp &= \left( \e_{2a-1}^+ \mp i\, \e_{2a-1}^+Ê\e_{2a-1}^- \right) = \mp \, i \, L_{2a-1}^\pm, &
L^\pm_{2a}Ê\,e_{2a}^\perp &= \left( \e_{2a}^+ \mp \e_{2a}^+Ê \e_{2a}^-\right) = \mp L_{2a}^\mp, \\
M^\pm_{2a-1} \,e_{2a-1}^\perp &= \left( -\e_{2a-1}^- \pm i\, \e_{2a-1}^- \e_{2a-1}^+Ê\right) = \pm \, i \, M_{2a-1}^\pm , &
M^\pm_{2a}Ê\,e_{2a}^\perp &= \left( - \e_{2a}^- \pm \e_{2a}^-Ê\e_{2a}^+ \right) = \pm \, M_{2a}^\mp.
\end{align*}}
We may indeed summarize this as
$$
F_{2a-1}\,e^\perp_{2a-1} = (-1)^{|F_{2a-1}| + 1} \,i\, F_{2a-1}, \qquad
F_{2a}\,e^\perp_{2a} = (-1)^{|F_{2a}|+1} \, \widetilde{F}_{2a}.
$$
From this, it follows that
$$
\widetilde{F}_{2a-1}\, e_{2a-1}^\perp = (-1)^{|F_{2a-1}|} \, i\, \widetilde{F}_{2a-1}, \qquad
\widetilde{F}_{2a}\,e^\perp_{2a} = (-1)^{|F_{2a}|} \, F_{2a}. 
$$
Hence, for $F_a \in \left\{ L_a^\pm, M_a^\pm \right\}$, we have
\begin{align*}
F_{2a-1} \, F_{2a} \, e_{2a-1}^\perp \, e_{2a}^\perp &= F_{2a-1}Ê\, e_{2a-1}\, \widetilde{F}_{2a}\, e_{2a}^\perp = (-1)^{|F_{2a}| + |F_{2a}| + 1}Ê\, i \, F_{2a-1} \,F_{2a}.
\end{align*}
Also important to note is that $e^\perp_a \, e_a \, L_a^\pm = L_a^\pm$ and $e^\perp_a\, e_a \, M^\pm_a = -M_a^\pm$ so for the idempotent $F = \prod_{s = 1}^m F_s $, we find that 
$$
V_{a,b} \, F = -e_a^\perp \, e_a \, e_b^\perp\, e_b\, F = (-1)^{1 + \| F_a \|Ê+ \| F_b \| } \, F.
$$
We thus get, for $F = \prod_{s=1}^m F_s$, that 
\begin{align*}
& V_{2a-1,2b-1} \, F \, e_{2a-1}^\perp \, e_{2b-1}^\perp = (-1)^{1 + \|F_{2a-1} \|Ê+ \|F_{2b-1} \|}Ê\, F_1 \, F_2 \ldots F_m \, e_{2a-1}^\perp \, e_{2b-1}^\perp \\
&= (-1)^{1 + \|F_{2a-1} \|Ê+ \|F_{2b-1} \|}Ê\, F_1 \ldots F_{2a-2}Ê\, F_{2a-1}Ê\, e_{2a-1}^\perp \, \widetilde{F}_{2a} \, \ldots \widetilde{F}_{2b-1}Ê\, e_{2b-1}^\perp \, F_{2b}Ê\, F_{2b+1} \ldots F_m \\
&= (-1)^{|F_{2a-1}| + |F_{2b-1}| + \|F_{2a-1} \|Ê+ \|F_{2b-1} \|}Ê\, i^2 \, F_1 \, F_2 \ldots F_{2a-2}Ê\, F_{2a-1} \, \widetilde{F}_{2a} \, \ldots \widetilde{F}_{2b-1}Ê\, F_{2b}Ê\, F_{2b+1} \ldots F_m \\
&= (-1)^{|F_{2a-1}| + |F_{2b-1}| + \|F_{2a-1} \|Ê+ \|F_{2b-1} \| + 1} \, F^{2a,2b-1}.
\end{align*}
Analogously, we find that 
\begin{align*}
& V_{2a-1,2b} \, F \, e_{2a-1}^\perp \, e_{2b}^\perp = (-1)^{1 + \|F_{2a-1} \|Ê+ \|F_{2b} \|}Ê\, F_1 \, F_2 \ldots F_m \, e_{2a-1}^\perp \, e_{2b}^\perp \\
&= (-1)^{1 + \|F_{2a-1} \|Ê+ \|F_{2b} \|}Ê\, F_1 \ldots F_{2a-2}Ê\, F_{2a-1}Ê\, e_{2a-1}^\perp \, \widetilde{F}_{2a} \, \ldots \widetilde{F}_{2b}Ê\, e_{2b}^\perp \, F_{2b+1}Ê\, F_{2b+2} \ldots F_m \\
&= (-1)^{|F_{2a-1}| + |F_{2b}| + \|F_{2a-1} \|Ê+ \|F_{2b} \|}Ê\, i \, F_1 \, F_2 \ldots F_{2a-2}Ê\, F_{2a-1} \, \widetilde{F}_{2a} \, \ldots \widetilde{F}_{2b-1}Ê\, F_{2b}Ê\, F_{2b+1} \ldots F_m \\
&= (-1)^{|F_{2a-1}| + |F_{2b}| + \|F_{2a-1} \|Ê+ \|F_{2b} \|} \, i\, F^{2a,2b-1}.
\end{align*}
Also 
\begin{align*}
& V_{2a,2b-1} \, F \, e_{2a}^\perp \, e_{2b-1}^\perp = (-1)^{1 + \|F_{2a} \|Ê+ \|F_{2b-1} \|}Ê\, F_1 \, F_2 \ldots F_m \, e_{2a}^\perp \, e_{2b-1}^\perp \\
&= (-1)^{1 + \|F_{2a} \|Ê+ \|F_{2b-1} \|}Ê\, F_1 \ldots F_{2a-1}Ê\, F_{2a}Ê\, e_{2a}^\perp \, \widetilde{F}_{2a+1} \, \ldots \widetilde{F}_{2b-1}Ê\, e_{2b-1}^\perp \, F_{2b}Ê\, F_{2b+1} \ldots F_m \\
&= (-1)^{|F_{2a}| + |F_{2b-1}| + \|F_{2a} \|Ê+ \|F_{2b-1} \|}Ê\, i \, F_1 \, F_2 \ldots F_{2a-1}Ê\, \widetilde{F}_{2a} \, \widetilde{F}_{2a+1} \, \ldots \widetilde{F}_{2b-1}Ê\, F_{2b}Ê\, F_{2b+1} \ldots F_m \\
&= (-1)^{|F_{2a}| + |F_{2b-1}| + \|F_{2a} \|Ê+ \|F_{2b-1} \|} \, i \, F^{2a,2b-1}.
\end{align*}
Finally
\begin{align*}
& V_{2a,2b} \, F \, e_{2a}^\perp \, e_{2b}^\perp = (-1)^{1 + \|F_{2a} \|Ê+ \|F_{2b} \|}Ê\, F_1 \, F_2 \ldots F_m \, e_{2a}^\perp \, e_{2b}^\perp \\
&= (-1)^{1 + \|F_{2a} \|Ê+ \|F_{2b} \|}Ê\, F_1 \ldots F_{2a-1}Ê\, F_{2a}Ê\, e_{2a}^\perp \, \widetilde{F}_{2a+1} \, \ldots \widetilde{F}_{2b}Ê\, e_{2b}^\perp \, F_{2b+1} \ldots F_m \\
&= (-1)^{|F_{2a}| + |F_{2b}| + \|F_{2a} \|Ê+ \|F_{2b} \|}Ê\, F_1 \, F_2 \ldots F_{2a-1}Ê\, \widetilde{F}_{2a} \, \widetilde{F}_{2a+1} \, \ldots \widetilde{F}_{2b-1}Ê\, F_{2b} \, F_{2b+1} \ldots F_m \\
&= (-1)^{|F_{2a}| + |F_{2b}| + \|F_{2a} \|Ê+ \|F_{2b} \|} \, F^{2a,2b-1}.
\end{align*}
\end{proof}

Consider the basic monogenic functions 
$$
g_{2k} = \left( \left( \xi_2 - \xi_1 \right) \left( \xi_2 + \xi_1 \right) \right)^k, \qquad g_{2k+1} = \left( \xi_2 - \xi_1 \right)  \left( \left( \xi_2 + \xi_1 \right) \left( \xi_2 - \xi_1 \right) \right)^k.
$$
From now on we denote $(k)'_+ = \left( k+ \frac{1}{2},\frac{1}{2},\ldots,\frac{1}{2}\right)$ and $(k)'_- = \left( k+ \frac{1}{2},\frac{1}{2},\ldots,-\frac{1}{2}\right)$. We will show under which conditions on the idempotent $F =Ê\prod_{s=1}^n F_s$, the space $\textup{span}_{\mathbb{C}} \left\{ g_k \, F \right\}$ is a weight space of $\frak{h}$ with weight $(k)'_+$ resp. $(k)'_-$.

\begin{lemma}\label{lem:WeightVector}
The polynomial $g_k \, F \in \mathcal{M}_k$, $F = \prod_{s=1}^m F_s$ with $F_s \in \left\{ L_s^\pm, M_s^\pm \right\}$, is a weight vector of $\gso(m,\mathbb{C})$ with 
\begin{itemize}
\item weight $(k)'_+$ when $k+ |F_1| + |F_2| + \|F_1 \|Ê+ \|F_2 \|$ is even and $\|ÊF_{2a-1}\|Ê+ \|F_{2a}\| + |F_{2a-1}|Ê+ |F_{2a}|$ is even for $2 \leqslant a \leqslant n$.

\item weight $(k)'_-$ when $k+ |F_1| + |F_2| + \|F_1 \|Ê+ \|F_2 \|$ is even, $\|ÊF_{2a-1}\|Ê+ \|F_{2a}\| + |F_{2a-1}|Ê+ |F_{2a}|$ is even for $2 \leqslant a \leqslant n-1$ and $\|ÊF_{2n-1}\|Ê+ \|F_{2n}\| + |F_{2n-1}|Ê+ |F_{2n}|$ is odd.\end{itemize}
\end{lemma}
\begin{proof}
We consider the action of the Cartan subalgebra-elements $H_s$, $1 \leqslant s \leqslant n$, on the $g_k\,F$. Since $g_k$ only contains $\xi_1$ and $\xi_2$, we will first consider $H_1$:
\begin{align*}
H_1 \left( g_k \, F \right) &= i\, V_{12} \left( L_{12} - \frac{1}{2} \right) g_kÊ\, F \, e_1^\perp \, e_2^\perp.
\end{align*}
We will also denote
$$
f_{2k} = \left( \left( \xi_2 + \xi_1 \right) \left( \xi_2 - \xi_1 \right) \right)^k, \qquad
f_{2k+1} = \left( \xi_2 + \xi_1 \right)  \left( \left( \xi_2 - \xi_1 \right) \left( \xi_2 + \xi_1 \right) \right)^k.
$$
In \cite{} it was established that $\p_j \,g_k = (-1)^j \, k \, f_{k-1}$:  
$$
L_{12} \, g_k = \left( \xi_1 \, \p_2 + \xi_2 \, \p_1 \right) g_k = k \left( \xi_1 - \xi_2 \right) f_{k-1} = -k \left( \xi_2 - \xi_1 \right) f_{k-1} = -k \, g_k. 
$$
We thus get that 
\begin{equation*}
H_1 \left( g_k \, F \right) = i \left( -k - \frac{1}{2} \right) V_{12} \, g_kÊ\, F \, e_1^\perp \, e_2^\perp.
\end{equation*}
Now we will show that $V_{12} \, g_k = (-1)^k \, g_kÊ\, V_{12}$: consider again $V_{12}Ê= - e_1^\perp e_1 e_2^\perp e_2$, then  
\begin{align*}
e_1^\perp e_1 \,\xi_1 &= \left( \e_1^+ \e_1^- - \e_1^- \e_1^+ \right)  \left( X_1^+ \, \e_1^- + X_1^- \, \e_1^+ \right) = \left( - X_1^+ \, \e_1^- + X_1^- \, \e_1^+ \right) \\
&= \left( X_1^+ \, \e_1^- + X_1^- \, \e_1^+ \right) \left( - \e_1^+ \e_1^- + \e_1^- \e_1^+ \right) = - \xi_1 \, e_1^\perp e_1, \\
e_1^\perp e_1 \, \xi_2 &= \xi_2 \,  e_1^\perp e_1.
\end{align*}
Hence
\begin{align*}
V_{12} \left( \xi_2 \pm \xi_1 \right) &= - e_1^\perp e_1 \, e_2^\perp e_2 \left( \xi_2 \pm \xi_1 \right) = \left( \xi_2 \pm \xi_1 \right) e_1^\perp e_1 \, e_2^\perp e_2 = - \left( \xi_2 \pm \xi_1 \right) V_{12}, \\
V_{12} \left( \left( \xi_2 - \xi_1 \right) \left( \xi_2 + \xi_1 \right) \right) &= \left( \left( \xi_2 - \xi_1 \right) \left( \xi_2 + \xi_1 \right) \right) V_{12}
\end{align*}
and so $V_{12}Ê\, g_k = (-1)^k \, g_k \, V_{12}$. Applying this, we find that 
\begin{align*}
H_1 \left( g_k \, F \right) &= (-1)^{k+1} \, i \left( k+ \frac{1}{2} \right)  \, g_k \, V_{12}Ê\, F \, e_1^\perp\,e_2^\perp \\
&= (-1)^{k+ |F_1| + |F_2| + \|F_1 \|Ê+ \|F_2 \|} \left( k+ \frac{1}{2} \right)  \, g_k \, F.
\end{align*}
To be a weight vector with weight $k + \frac{1}{2}$, it must hold that 
$$
k+ |F_1| + |F_2| + \|F_1 \|Ê+ \|F_2 \| \text{ is even}.
$$
We thus find $8$ possible combinations for $F_1\,F_2$:
\begin{itemize}
\item $k$ even: 
$$
F_1 \, F_2 \in \left\{ L_1^+\, L_2^+,\  L_1^- \, L_2^-, \ L_1^+\,M_2^+, \ L_1^- \,M_2^-, \ M_1^+\,L_2^+, \ M_1^+ \,  M_2^+, \ M_1^- \, L_2^-, \ M_1^- \, M_2^- \right\}.
$$

\item $k$ odd:
$$
F_1 \, F_2 \in \left\{ L_1^+\, L_2^-,\  L_1^- \, L_2^+, \ L_1^+\,M_2^-, \ L_1^- \,M_2^+, \ M_1^+\,L_2^-, \ M_1^+ \,  M_2^-, \ M_1^- \, L_2^+, \ M_1^- \, M_2^+ \right\}.
$$
\end{itemize}

Next, we consider $H_a$, $1 < a \leqslant n$. Since the generator $g_k$ only contains $\xi_1$ and $\xi_2$, it vanishes under the action of $L_{2a-1,2a}$. Note that $V_{2a-1,2a} \, g_k = g_k \, V_{2a-1,2a}$ since $g_k$ contains only $\e_1^\pm$ and $\e_2^\pm$. Thus
\begin{align*}
H_a \left( g_k \, F \right) &= - \frac{i}{2}\, V_{2a-1,2a} \, g_kÊ\, F \, e_{2a-1}^\perp \,e_{2a}^\perp = - \frac{i}{2}\,  g_kÊ\, V_{2a-1,2a} \,F \, e_{2a-1}^\perp \,e_{2a}^\perp \\
&= (-1)^{\|ÊF_{2a-1}\|Ê+ \|F_{2a}\| + |F_{2a-1}|Ê+ |F_{2a}| + 1}\,i \, \frac{i}{2}\,  g_kÊ\, F \\Ê
&= (-1)^{\|ÊF_{2a-1}\|Ê+ \|F_{2a}\| + |F_{2a-1}|Ê+ |F_{2a}|}\, \frac{1}{2}\,  g_kÊ\, F.
\end{align*}
This equals $+\frac{1}{2} \, g_k \, F$ when $\|ÊF_{2a-1}\|Ê+ \|F_{2a}\| + |F_{2a-1}|Ê+ |F_{2a}|$ is even and $-\frac{1}{2}\, g_k\,F$ when $\|ÊF_{2a-1}\|Ê+ \|F_{2a}\| + |F_{2a-1}|Ê+ |F_{2a}|$ is odd. We may thus conclude that the statement holds. 

We find that $\|ÊF_{2a-1}\|Ê+ \|F_{2a}\| + |F_{2a-1}|Ê+ |F_{2a}|$ is even for $F_{2a-1}Ê\, F_{2a}$ in 
$$
\left\{ L_{2a-1}^+\, L_{2a}^+, L_{2a-1}^- \, L_{2a}^-, L_{2a-1}^+\,M_{2a}^+, L_{2a-1}^- \,M_{2a}^-, M_{2a-1}^+\,L_{2a}^+, M_{2a-1}^+ \, M_{2a}^+, M_{2a-1}^- \, L_{2a}^-, M_{2a-1}^- \, M_{2a}^- \right\}
$$
and odd for $F_{2a-1}Ê\, F_{2a} $ in 
$$
\left\{ L_{2a-1}^+\, L_{2a}^-, L_{2a-1}^- \, L_{2a}^+, L_{2a-1}^+ \, M_{2a}^-, L_{2a-1}^- \,M_{2a}^+, M_{2a-1}^+\,L_{2a}^-, M_{2a-1}^+ \, M_{2a}^-, M_{2a-1}^- \, L_{2a}^+, M_{2a-1}^- \, M_{2a}^+ \right\}.
$$
\end{proof}

\begin{rem}
In particular, we find that $g_{2k}Ê\, \prod_{s=1}^m L_s^+$ respectively $g_{2k+1} \, L_1^+ \, L_2^-\, \prod_{s=3}^m L_s^+$ are weight vectors of $\frak{h}$ in $\mathcal{M}_{2k}$ resp. $\mathcal{M}_{2k+1}$ of weight $(2k)'_+$ resp. $(2k+1)'_+$. 
\end{rem}

\begin{cor}
There are $2^{2m-n}$ weight vectors $g_k \, F$, with $F$ one of the above mentioned idempotents, of weight $(k)'_+$ and $2^{2m-n}$ weight vectors $g_k \, F$, with $F$ one of the above mentioned idempotents, with weight $(k)'_-$.
\end{cor}

\begin{proof}
To obtain weight $(k)'_+$, one has eight choices for each factor $F_{2s-1} \, F_{2s}$ in $F$, $1 \leqslant s \leqslant n$. We thus get $8^{n} = 2^{3n} = 2^{4n-n}Ê= 2^{2m-n}$ choices for the idempotent $F$. The same count holds for the weight $(k)'_-$.  
\end{proof}

We will now show that the weight vectors, defined in Lemma \ref{lem:WeightVector} are actually highest weight vectors, i.e. that they vanish under the action of all positive roots. 
\begin{lemma}
The polynomials $g_k \, F$, with 
\begin{itemize}
\item $k+ |F_1| + |F_2| + \|F_1 \|Ê+ \|F_2 \|$ even
\item $\|ÊF_{2a-1}\|Ê+ \|F_{2a}\| + |F_{2a-1}|Ê+ |F_{2a}|$ even, $\forall \, 2 \leqslant a \leqslant n-1$, and 
\item $\|ÊF_{2n-1}\|Ê+ \|F_{2n}\| + |F_{2n-1}|Ê+ |F_{2n}|$ even resp. odd
\end{itemize}
are highest weight spaces with highest weight $(k)'_+$ resp. $(k)'_-$, i.e. 
$$
H_a \left( g_k\,F\right) = \left( \delta_{1a}Ê\left( k + \frac{1}{2}Ê\right) + \frac{1}{2}Ê\sum_{j \neq 1} \delta_{ja} \right) g_k \, F
$$
and 
$$
X_{a,b}Ê\left( g_k \, F \right) = 0, \; \forall (a,b), \, a < b, \qquad
Y_{a,b} \left( g_k \, F \right) = 0, \; \forall (a,b), \, a \neq b.
$$
\end{lemma}
\begin{proof}
We have already shown that these $g_k \, F$ are weight vectors with weight $(k)'_\pm$. We now show that $g_k\,F$ vanishes under the action of $X_{a,b}$, $a< b$, and $Y_{a,b}$, $a \neq b$. Note that $X_{a,b}\left( g_kÊ\, F \right)$ denotes the action of the operator $X_{a,b}$ on $g_kÊ\,F$; this is not a multiplication. 

\smallskip We first consider the action of $X_{a,b}$ on $g_k\, F$. We make a distinction between $a = 1$ and $a \neq 1$. First let $1 = a < b$, then 
\begin{align*}
2\, & X_{1,b} \, g_k \, F = \left(dR(e_{1,2b-1}) + i \, dR(e_{1,2b}) - i \, dR(e_{2,2b-1}) + dR(e_{2,2b})Ê\right) g_k \, F \\
&= V_{1,2b-1} \left(\xi_{2b-1} \, \p_1 - \frac{1}{2}Ê\right) g_kÊ\, F \,e_1^\perp \, e_{2b-1}^\perp + i \, V_{1,2b} \left( \xi_{2b}\,\p_1 - \frac{1}{2}Ê\right) g_k \, F \,e_1^\perp\,e_{2b}^\perp \\
&\phantom{=} - i \, V_{2,2b-1} \left(\xi_{2b-1} \, \p_2 - \frac{1}{2}Ê\right) g_kÊ\, F \,e_2^\perp \,e_{2b-1}^\perp + V_{2,2b} \left(\xi_{2b}\,\p_2 - \frac{1}{2}Ê\right) g_kÊ\, F \,e_2^\perp\, e_{2b}^\perp \\
&= V_{1,2b-1} \left( - k\, \xi_{2b-1}\, f_{k-1} - \frac{1}{2}Ê\, g_kÊ\right) F \,e_1^\perp\, e_{2b-1}^\perp + i \, V_{1,2b} \left( -k \, \xi_{2b}\,f_{k-1} - \frac{1}{2}Ê\, g_k \right) F \,e_1^\perp \,e_{2b}^\perp \\
&\phantom{=} - i \, V_{2,2b-1} \left( k \, \xi_{2b-1} \, f_{k-1} - \frac{1}{2}Ê\, g_k \right) F \,e_2^\perp \, e_{2b-1}^\perp + V_{2,2b} \left( k\, \xi_{2b}\,f_{k-1} - \frac{1}{2}\, g_kÊ\right) F \, e_2^\perp \, e_{2b}^\perp.
\end{align*}
Now we use 
$$
V_{1,2b-1} \, \xi_{2b-1}Ê= -\xi_{2b-1}Ê\, V_{1,2b-1}, \qquad V_{1,2b} \, \xi_{2b}Ê= -\xi_{2b}Ê\, V_{1,2b}.
$$
Furthermore, since for $b \neq 1,2$, 
\begin{align*}
V_{1,b} \left( \xi_2 \pm \xi_1 \right) &= \left( \xi_2 \mp \xi_1 \right) V_{1,b}, \\
V_{2,b}Ê\left( \xi_2 \pm \xi_1 \right) &= \left( -\xi_2 \pm \xi_1 \right) V_{2,b} = -\left( \xi_2 \mp \xi_1 \right) V_{2,b},
\end{align*}
we find, for $j \in \left\{ 2b-1, 2b\right\} $:
\begin{align*}
V_{1,j} \, g_kÊ&= V_{1,j} \left( \xi_2 - \xi_1 \right) \left( \xi_2 + \xi_1 \right) \left( \xi_2 - \xi_1 \right) \ldots = \left( \xi_2 + \xi_1 \right) \left( \xi_2 - \xi_1 \right) \left( \xi_2 + \xi_1 \right) \ldots V_{1,j} = f_k \, V_{1,j},  \\
V_{1,j} \, f_kÊ&= V_{1,j} \left( \xi_2 + \xi_1 \right) \left( \xi_2 - \xi_1 \right) \left( \xi_2 + \xi_1 \right) \ldots = \left( \xi_2 - \xi_1 \right) \left( \xi_2 + \xi_1 \right) \left( \xi_2 - \xi_1 \right) \ldots V_{1,j} = g_k \, V_{1,j},  \\
V_{2,j} \, g_kÊ&= V_{2,j} \left( \xi_2 - \xi_1 \right) \left( \xi_2 + \xi_1 \right) \left( \xi_2 - \xi_1 \right) \ldots = (-1)^k\, f_k \, V_{2,j},  \\
V_{2,j} \, f_kÊ&= V_{2,j} \left( \xi_2 + \xi_1 \right) \left( \xi_2 - \xi_1 \right) \left( \xi_2 + \xi_1 \right) \ldots = (-1)^k \, g_k \, V_{2,j}.
\end{align*} 
We get that
\begin{align*}
2\, X_{1,b} \, g_k \, F &= \left( k\, \xi_{2b-1}\,g_{k-1} - \frac{1}{2}Ê\, f_kÊ\right) V_{1,2b-1}\, F \,e_1^\perp\, e_{2b-1}^\perp + i \, \left( k \, \xi_{2b} \, g_{k-1} - \frac{1}{2}Ê\, f_k \right) V_{1,2b}\, F \,e_1^\perp \,e_{2b}^\perp \\
&\phantom{=} - i \left( (-1)^{1+ k-1} \, k \, \xi_{2b-1} \, g_{k-1} - (-1)^k \, \frac{1}{2}Ê\, f_k \right) V_{2,2b-1} \, F \,e_2^\perp \, e_{2b-1}^\perp \\
&\phantom{=} + \left( (-1)^{1 + k-1} \, k\, \xi_{2b}\, g_{k-1} - \frac{1}{2}\, (-1)^k \, f_kÊ\right) V_{2,2b}\, F \, e_2^\perp \, e_{2b}^\perp.
\end{align*}
We now use that 
\begin{align*}
V_{1,2b-1} \,F \, e_1^\perp \, e_{2b-1}^\perp &= (-1)^{|F_1| + |F_{2b-1}| + \|F_1 \|Ê+ \|F_{2b-1} \| + 1} \, F^{2,2b-1}, \\
V_{1,2b} \, F \, e_{1}^\perp \, e_{2b}^\perp &= (-1)^{|F_{1}| + |F_{2b}| + \|F_{1} \|Ê+ \|F_{2b} \|} \, i\, F^{2,2b-1}, \\
V_{2,2b-1} \, F \, e_{2}^\perp \, e_{2b-1}^\perp &= (-1)^{|F_{2}| + |F_{2b-1}| + \|F_{2} \|Ê+ \|F_{2b-1} \|} \, i \, F^{2,2b-1}, \\
V_{2,2b} \, F \, e_{2}^\perp \, e_{2b}^\perp &= (-1)^{|F_{2}| + |F_{2b}| + \|F_{2} \|Ê+ \|F_{2b} \|} \, F^{2,2b-1}.
\end{align*}
This results in 
\begin{align*}
& 2\, X_{1,b} \, g_k \, F = (-1)^{|F_1| + \|F_1 \|Ê+ |F_{2b-1}| + \|F_{2b-1} \| }Ê\\
&\phantom{=} \left( -k\, \xi_{2b-1}\,g_{k-1} + \frac{1}{2}Ê\, f_k + (-1)^{|F_{2b-1}| + \|F_{2b-1} \| + |F_{2b}| + \|F_{2b} \|} \left( - k \, \xi_{2b} \, g_{k-1} + \frac{1}{2}Ê\, f_k \right) \right. \\
&\phantom{=} + \left( (-1)^k \, k \, \xi_{2b-1} \, g_{k-1} - (-1)^k \, \frac{1}{2}Ê\, f_k \right) (-1)^{|F_1|Ê+ \| F_1 \| + |F_{2}| + \|F_{2} \|}  \\
&\phantom{=} \left. + \left( (-1)^k \, k\, \xi_{2b}\, g_{k-1} - \frac{1}{2}\, (-1)^k \, f_kÊ\right)  (-1)^{|F_1|Ê+ \| F_1 \| + |F_{2}| + \|F_{2} \| + |F_{2b-1}| + |F_{2b}| + \|F_{2b-1} \|Ê+ \|F_{2b} \|} \right) F^{2,2b-1}.
\end{align*}
We thus see that this vanishes when 
$$
k+ |F_1|Ê+ \| F_1 \| + |F_{2}| + \|F_{2} \|
$$
is even.

\medskip For $1 < a < b \leqslant n$ we get that 
\begin{align*}
2\,& X_{a,b} \, g_k \, F = -\frac{1}{2}Ê\left( V_{2a-1,2b-1} \, g_kÊ\, F \,e_{2a-1}^\perp \, e_{2b-1}^\perp  + i \, V_{2a-1,2b} \, g_kÊ\, F \, e_{2a-1}^\perp \,e_{2b}^\perp \right.  \\
&\phantom{===} \left. - i \, V_{2a,2b-1} \, g_kÊ\, F \, e_{2a}^\perp \,e_{2b-1}^\perp + V_{2a,2b} \, g_kÊ\, F \, e_{2a}^\perp \, e_{2b}^\perp \right) \\
&= -\frac{1}{2}Ê\, g_k \left( (-1)^{|F_{2a-1}| + |F_{2b-1}| + \|F_{2a-1} \|Ê+ \|F_{2b-1} \| + 1} + (-1)^{|F_{2a-1}| + |F_{2b}| + \|F_{2a-1} \|Ê+ \|F_{2b} \| + 1} \right.  \\
&\phantom{===} \left. + (-1)^{|F_{2a}| + |F_{2b-1}| + \|F_{2a} \|Ê+ \|F_{2b-1} \|} + (-1)^{|F_{2a}| + |F_{2b}| + \|F_{2a} \|Ê+ \|F_{2b} \|} \right) F^{2a,2b-1} \\
&= -\frac{1}{2}Ê\, g_k\, (-1)^{|F_{2a-1}| + \|F_{2a-1} \|Ê+ |F_{2b-1}| + \|F_{2b-1} \|} \\
&\phantom{=} \left(  -1 + (-1)^{|F_{2b-1}| + \| F_{2b-1}\| + |F_{2b}| + \|F_{2b} \| + 1} \right.  \\
&\phantom{=} \left. + (-1)^{ |F_{2a-1}|Ê+ \|F_{2a-1} \| + |F_{2a}| + \|F_{2a} \|} + (-1)^{ |F_{2a-1}|Ê+ \|F_{2a-1} \| + |F_{2a}| + \|F_{2a} \|+ |F_{2a}| + \|F_{2a} \|Ê+ |F_{2b}| + \|F_{2b} \|} \right) F^{2a,2b-1}.
\end{align*}
This will be zero when  $ |F_{2a-1}|Ê+ \|F_{2a-1} \| + |F_{2a}| + \|F_{2a} \|$ is even, and this for all $2 \leqslant a \leqslant n-1$. 

\medskip Note that:
\begin{align*}
X_{a,b} &= \frac{1}{2} \left( dR(e_{2a-1,2b-1}) + i \, dR(e_{2a-1,2b}) - i \, dR(e_{2a,2b-1}) + dR(e_{2a,2b})Ê\right), \\
Y_{a,b} &= \frac{1}{2} \left( dR(e_{2a-1,2b-1}) - i \, dR(e_{2a-1,2b}) - i \, dR(e_{2a,2b-1}) - dR(e_{2a,2b})Ê\right).\end{align*}
If we apply the appropriate change of sign in the second and last term of previous calculations, we immediately get that $Y_{a,b}(g_k \, F) = 0$ for $a < b$. Since $Y_{a,b} = - Y_{b,a}$, this will also be zero for $a > b$.
\end{proof}

\begin{rem}
In particular, the polynomial $g_{2k}Ê\, \prod_{s=1}^m L_s^+$ and $g_{2k+1}Ê\, L_1^+ \, L_2^- \, \prod_{s=3}^m L_s^+$ are highest weight vectors with weight $(2k)'_+$ resp. $(2k+1)'_+$. 
\end{rem}

\begin{rem}
The dimension of $\left(k\right)'_\pm$ is (see \cite{FH}) 
$$
2^{n-1} \, \binom{k+m-2}{k}.
$$
As the dimension of $\mathcal{M}_k$ equals 
$$
2^{2m} \, \binom{k+m-2}{k} = 2^{4n} \, \binom{k+m-2}{k}
$$
and as we found $2^{3n}$ isomorphic copies of $\left(k\right)'_+$ combined with $2^{3n}$ copies of $\left(k\right)'_-$, the space $\mathcal{M}_k$ is fully decomposed in $2^{3n}$ copies of $\left(k\right)'_+$ and $2^{3n}$ copies of $\left(k \right)'_-$.
\end{rem}

\begin{defi}
We define the positive resp. negative spinorspace $\mathbb{S}_{2n}^\pm$ as the image under $\gso(m,\mathbb{C})$ of the idempotents $\prod_{s=1}^m L_s^+$, resp. $\left( \prod_{s=1}^{m-1} L_s^+ \right) L_m^-$:
$$
\mathbb{S}_{2n}^+ = \gso(m,\mathbb{C}) \left( \textup{span}_{\mathbb{C}} \left\{ L_1^+ \,L_2^+ \ldots \, L_{2n-1}^+ \, L_{2n}^+ Ê\right\} \right)
$$
and 
$$
\mathbb{S}_{2n}^- = \gso(m,\mathbb{C}) \left( \textup{span}_{\mathbb{C}} \left\{ L_1^+ \,L_2^+ \ldots \, L_{2n-1}^+ \, L_{2n}^-  \right\} \right).
$$
\end{defi}
The elements $L_1^+ \,L_2^+ \ldots \, L_{2n-1}^+ \, L_{2n}^+$, resp. $L_1^+ \,L_2^+ \ldots \, L_{2n-1}^+ \, L_{2n}^-$ are highest weight vectors with weight $(0)'_+ = \left( \frac{1}{2}, \ldots, \frac{1}{2}\right)$ resp. $(0)'_- = \left(\frac{1}{2}, \ldots, \frac{1}{2},-\frac{1}{2}\right)$ and they thus generate irreducible representations with the same weight. 

\begin{example}
Let $m = 4$ (i.e. $n=2$) and consider $L = L_1^+ \,L_2^+ \, L_3^+ \, L_4^+$. The Lie algebra $\gso(4,\mathbb{C})$ is given in this context by 
$$
\textup{span}_{\mathbb{C}} \left\{ dR(e_{12}), dR(e_{13}), dR(e_{14}), dR(e_{23}), dR(e_{24}), dR(e_{34}) \right\}.
$$
The elements $dR(e_{12})$ and $dR(e_{34})$ return $L$ up to complex constant. The other four rotations give us (up to a complex constant) the idempotent $L_1^+\,L_2^- \, L_3^-\, L_4^+$. Hence
$$
\mathbb{S}^+_4 = \textup{span}_{\mathbb{C}} \left\{ L_1^+ \,L_2^+ \, L_3^+ \, L_4^+, L_1^+ \,L_2^- \, L_3^- \, L_4^+\right\}.
$$
Starting from $L_1^+ \,L_2^+ \, L_3^+ \, L_4^-$, the rotations $dR(e_{13})$, $dR(e_{14})$, $dR(e_{23})$ and $dR(e_{24})$ lead us to the idempotent $L_1^+ \,L_2^- \, L_3^- \, L_4^-$ which shows that 
$$
\mathbb{S}^-_4 = \textup{span}_{\mathbb{C}} \left\{ L_1^+ \,L_2^+ \, L_3^+ \, L_4^-, L_1^+ \,L_2^- \, L_3^- \, L_4^-\right\}.
$$
The (positive/negative) spinorspace $\mathbb{S}^\pm_{2n}$ is $2^{n-1}$-dimensional. 
\end{example}


In general, the elements $dR(e_{2a-1,2a})$ acting on an idempotent return the same idempotent up to a multiplicative complex factor. Since, for $1 \leqslant a < b \leqslant n$:
\begin{align*}
V_{2a-1,2b-1} \, L \, e_{2a-1}^\perp \, e_{2b-1}^\perp &= -L^{a,b}, &
V_{2a-1,2b}Ê\, L \, e_{2a-1}^\perp \, e_{2b}^\perp &= i\,L^{a,b}, \\
V_{2a,2b-1} \, L \, e_{2a}^\perp \, e_{2b-1}^\perp &= i\, L^{a,b}, &
V_{2a,2b} \, L \, e_{2a}^\perp \, e_{2b}^\perp &= L^{a,b}.
\end{align*}
with $L^{a,b} = L_1^+ \, L_2^+ \, \ldots L_{2a-1}^+ \, L_{2a}^- \ldots L_{2b-1}^- \, L_{2b}^+ \ldots L_{2n-1}^+ \, L_{2n}^+$, we see that $dR(e_{a,b})$ acting on  
$$
L = L_1^+ \, L_2^+ \ldots L_{2n-1}^+ \, L_{2n}^+
$$
changes the sign of an even number of $L_a$'s. The operator always leaves $L_1^+$ and $L_{2n}^+$ invariant. The resulting idempotent will always have an even number of minus-signs. Starting from the idempotent $L$ with all plus-signs, we thus get all possible idempotents of the following form:
$$
L_1^+ \underbrace{ \, . \quad . \, } \; \underbrace{ \, . \quad . \, } \; \ldots \; \underbrace{ \, . \quad . \, }\; L_{2n}^+.
$$
where each place $\; \underbrace{ \, . \quad . \, } \; $ consists of either $L_{2a}^+ \, L_{2a+1}^+$ or $L_{2a}^- \, L_{2a+1}^-$, $1 \leqslant a \leqslant n-1$. We get $2^{n-1}$ spinors belonging to the positive spinorspace and we have the following weight space decomposition
$$
\mathbb{S}_{2n}^+ = \bigoplus V_{\left( \pm \frac{1}{2},  \pm \frac{1}{2},\ldots, \pm \frac{1}{2}\right)},
$$
where the sum goes over all weights with an even number of minus-signs. The highest weight remains $\left( \frac{1}{2}, \ldots, \frac{1}{2}\right)$ and the highest weight vector is $L$. 

\medskip Starting from $L_1^+ \, L_2^+ \ldots L_{2n-1}^+ \, L_{2n}^-$, we will generate all possible idempotents of the following form:
$$
L_1^+ \underbrace{ \, . \quad . \, } \; \underbrace{ \, . \quad . \, } \; \ldots \; \underbrace{ \, . \quad . \, }\; L_{2n}^-.
$$
where each place $\; \underbrace{ \, . \quad . \, } \; $ consists of either $L_{2a}^+ \, L_{2a+1}^+$ or $L_{2a}^- \, L_{2a+1}^-$, $1 \leqslant a \leqslant n-1$. We thus also get $2^{n-1}$ spinors belonging to the negative spinorspace and the following weight space decomposition: 
$$
\mathbb{S}_{2n}^- = \bigoplus V_{\left( \pm \frac{1}{2},  \pm \frac{1}{2},\ldots, \pm \frac{1}{2}\right)},
$$
where the sum goes over all weights with an odd number of minus-signs. The highest weight is still $\left(\frac{1}{2}, \ldots, \frac{1}{2},-\frac{1}{2}\right)$ and the highest weight vector is $L_1^+ \, L_2^+ \ldots L_{2n-1}^+ \, L_{2n}^-$.

\subsection{Odd dimension $m=2n+1$}
We now extend the set of generators $H_a$, $X_{a,b}$, $Y_{a,b}$ and $Z_{a,b}$ of $\gso(m,\mathbb{C})$ with $2n$ mappings
\begin{align*}
U_a &= \frac{1}{\sqrt{2}}Ê\left( dR(e_{2a-1,m}) - i \, dR(e_{2a,m}) \right), \\
V_a &= \frac{1}{\sqrt{2}}Ê\left( dR(e_{2a-1,m}) + i \, dR(e_{2a,m}) \right),
\end{align*}
where $1 \leqslant a \leqslant n$. With the addition of these $2n$ mappings, we are again able to reconstruct all original $dR(e_{a,b})$'s since $\sqrt{2} \; dR(e_{2a-1,m}) = U_a + V_a$ and $-\sqrt{2} \,i \, dR(e_{2a,m}) = U_a - V_a$.

\medskip The classic commutator relations follow immediately. 
\begin{lemma}
For $1 \leqslant a, b \leqslant n$, it holds that
\begin{align*}
\left[ H_a, U_b \right] &= \delta_{ab} \, U_b = L_b(H_a) \, U_b, \\
\left[ H_a, V_b \right] &= -\delta_{ab}Ê\, V_b = -L_b(H_a) \, V_b.
\end{align*}
In particular, $U_b$ is a root vector corresponding to the positive root $L_b$ and $V_b$ is a root vector corresponding with the negative root $-L_b$, $\forall \, 1 \leqslant b \leqslant n$.
\end{lemma}

\begin{lemma}
The operators $U_c$ and $V_d$ satisfy the following additional commutator relations with $X_{a,b}$, $Y_{a,b}$ and $Z_{a,b}$, $1 \leqslant a, b, c,d \leqslant n$:
\begin{align*}
\left[ U_c, X_{a,b} \right]Ê&= -\delta_{cb}\,U_a, & 
\left[ V_c, X_{a,b} \right]Ê&= \delta_{ca}\,V_b, \\
\left[ U_c, Y_{a,b} \right]Ê&= 0, &
\left[ V_c, Y_{a,b} \right]Ê&= \delta_{ca}\,U_b - \delta_{cb} \, U_a, \\
\left[ U_c, Z_{a,b} \right]Ê&= -\delta_{cb}\,V_a + \delta_{ca} \, V_b, &
\left[ V_c, Z_{a,b} \right]Ê&= 0, \\
\left[ U_c, U_d \right]Ê&= -Y_{c,d}, \ c \neq d, & 
\left[ V_c, V_{d} \right]Ê&= -Z_{c,d}, \ c \neq d, \\
\left[ U_c, V_{d} \right]Ê&= \begin{cases} 
	-X_{c,d}, & c \neq d, \\
	-H_c, & c = d. 
	\end{cases}
\end{align*}
\end{lemma}
\begin{proof}
The statements follow immediately from the definitions of $U_c$ and $V_d$ and from the defining relations (\ref{eq:commutator_rule_dR}) which the operators $dR(e_{a,b})$ satisfy. 
\end{proof}

We now introduce four extra idempotents
$$
L_m^\pm = \left( \e_m^+ \e_m^- \pm i\, \e_m^+ \right), \qquad M_m^\pm = \left( \e_m^-\e_m^+ \pm \e_m^-\right)
$$
and denote
$$
L = \prod_{a=1}^n \left( L_{2a-1}^+ \, L_{2a}^+ \right) L_m^+, \qquad L' = L_1^+ \, L_2^-Ê\, \prod_{a=2}^n \left( L_{2a-1}^+ \, L_{2a}^+ \right) L_m^+.
$$
We will now show that the highest weight vectors of weight $(k)'_+$ from the even-dimensional setting are still highest weight vectors with weight $(k)'_+$ when we add one of the four possible extra factors to the idempotent $F$. 

\begin{lemma}
The weight vectors $g_k \, F$, $F = \prod_{s=1}^m F_s$ with $F_s \in \left\{ L_s^\pm, M_s^\pm \right\}$, such that 
\begin{itemize}
\item $k+ |F_1| + |F_2| + \|F_1 \|Ê+ \|F_2 \|$ even
\item $\|ÊF_{2a-1}\|Ê+ \|F_{2a}\| + |F_{2a-1}|Ê+ |F_{2a}|$ even, $\forall \, 2 \leqslant a \leqslant n$,
\end{itemize}
vanish under the operator $U_a$, $1 \leqslant a \leqslant n$, i.e. 
$$
U_a \left( g_kÊ\,F \right) = 0, \qquad \forall \, 1 \leqslant a \leqslant n.
$$
\end{lemma}
\begin{proof}
Consider
\begin{align*}
\sqrt{2}Ê\, U_a \left( g_kÊ\, F \right) &= \left( dR(e_{2a-1,m})Ê- i \, dR(e_{2a,m}Ê\right) g_kÊ\, F.
\end{align*}
Since $g_k$ contains only $\xi_1$ and $\xi_2$, we will make a distinction between $a = 1$ and $a \neq 1$. We start with assuming that $a=1$. Then 
\begin{align*}
\sqrt{2}Ê\,& U_1 \left( g_kÊ\, F \right) = i\, V_{1,m} \left( \xi_m \, \p_1 - \frac{1}{2} \right) g_kÊ\, F \,e_1^\perp  \, e_m^\perp - i^2 \, V_{2,m} \left( \xi_m \,\p_2 - \frac{1}{2}Ê\right) g_kÊ\, F \, e_2^\perp \, e_m^\perp \\
&= i\, V_{1,m} \left( -k \, \xi_m \, f_{k-1} - \frac{1}{2} \, g_k \right) F \,e_1^\perp  \, e_m^\perp + V_{2,m} \left( k\, \xi_m \, f_{k-1} - \frac{1}{2}Ê\, g_k \right) F \, e_2^\perp \, e_m^\perp.
\end{align*}
Now we again use that, for $j \neq 1,2$:
$$
V_{1,j} \, f_k = g_k \, V_{1,j}, \quad 
V_{2, j}Ê\,f_k = (-1)^k \, g_k \, V_{2,j}, \quad
V_{1,j}Ê\,g_k = f_k \, V_{1,j}, \quad
V_{2,j}Ê\,g_k = (-1)^k \, f_k \, V_{2,j}. 
$$ 
Hence
\begin{align*}
& \sqrt{2}Ê\,U_1 \left( g_kÊ\, F \right) \\
&= i\, \left( k \, \xi_m \, V_{1,m} \, f_{k-1} - \frac{1}{2} \, V_{1,m} \, g_k \right) F \,e_1^\perp  \, e_m^\perp + \left( -k\, \xi_m \, V_{2,m} \, f_{k-1} - \frac{1}{2}Ê\, V_{2,m} \, g_k \right) F \, e_2^\perp \, e_m^\perp \\
&= i\, \left( k \, \xi_m \, g_{k-1} - \frac{1}{2} \, f_k \right) V_{1,m} \, F \,e_1^\perp  \, e_m^\perp + \left( (-1)^{k-1+1} \, k\, \xi_m \, g_{k-1} - \frac{1}{2}Ê\, (-1)^k \, f_k \right) V_{2,m} \, F \, e_2^\perp \, e_m^\perp.
\end{align*}
We complete the proof by noting that 
\begin{align*}
V_{1,m}Ê\, F \, e_1^\perp \, e_m^\perp &= (-1)^{\|F_1\|Ê+ \|F_m\|Ê+ 1} \, F \, e_1^\perp \,e_m^\perp = (-1)^{\|F_1\|Ê+ \|F_m\|Ê+ 1} \, F_1 \, e_1^\perp \, \widetilde{F}_2 \ldots \widetilde{F}_{m-1}Ê\, \widetilde{F}_m \,e_m^\perp \\
&= (-1)^{\|F_1\|Ê+ |F_1|Ê+ \|F_m\| + |F_m|Ê} \, i^2 \, F_1 \, \widetilde{F}_2 \ldots \widetilde{F}_{m-1}Ê\, \widetilde{F}_m  \\
&= (-1)^{\|F_1\|Ê+ |F_1|Ê+ \|F_m\| + |F_m|Ê+ 1} \, F^{2,m}, \\ 
V_{2,m}Ê\, F \, e_2^\perp \, e_m^\perp &= (-1)^{\|ÊF_2\|Ê+ \|F_m \|Ê+ 1} \, F_1 \, F_2 \, e_2^\perp \, \widetilde{F}_3\, \widetilde{F}_4 \ldots \widetilde{F}_{m-1}Ê\, \widetilde{F}_m \,e_m^\perp \\
&= (-1)^{\|ÊF_2\| + |F_2|Ê+ \|F_m \| + |F_m|} \, i \, F^{2,m}.
\end{align*}
Thus $\sqrt{2}\,U_1(g_k \, F)$ will be zero since $k+ |F_1| + |F_2| + \|F_1 \|Ê+ \|F_2 \|$ is even.

\medskip  When $a \neq 1$, the action of $L_{2a-1,m}$ on $g_k\, F$ results in zero hence
\begin{align*}
& \sqrt{2}Ê\, U_a \left( g_kÊ\, F \right) = -\frac{i}{2}Ê\, V_{2a-1,m} \, g_kÊ\, F \, e_{2a-1}^\perp \, e_m^\perp - \frac{1}{2} \, V_{2a,m} \, g_kÊ\, F \, e_{2a}^\perp \, e_m^\perp \\
&= -\frac{1}{2}Ê\, g_kÊ\left( i \, V_{2a-1,m} \, F \, e_{2a-1}^\perp \, e_m^\perp + V_{2a,m} \, F \, e_{2a}^\perp \, e_m^\perp \right) \\
&= -\frac{1}{2}Ê\, g_kÊ\left( (-1)^{\|F_{2a-1}\| + |F_{2a-1}|Ê+ \|ÊF_m\| + |F_m|} \, i + (-1)^{\|F_{2a}\|Ê+ |F_{2a}| + \| F_m\| + |F_m| + 1}\, i \right)  \, F^{2a,m} \\
&= -\frac{i}{2}Ê\, g_k\,(-1)^{ \|F_{2a-1}\| + |F_{2a-1}|Ê+ \|ÊF_m\| + |F_m|}Ê\left( 1 + (-1)^{\|F_{2a-1}\| + |F_{2a-1}|+ \|F_{2a}\|Ê+ |F_{2a}| + 1} \right)  \, F^{2a,m}.
\end{align*}
This will be zero when $\|ÊF_{2a-1}\|Ê+ \|F_{2a}\| + |F_{2a-1}|Ê+ |F_{2a}|$ is even, $\forall \, 2 \leqslant a \leqslant n$. Note that the highest weight vectors of weight $(k)'_-$ will not vanish under the action of $U_n$. 
\end{proof}

\begin{cor}
The polynomials $g_k \, F$ with $F = \prod_{s=1}^m F_s$, $F_s \in \left\{ L_s^\pm, M_s^\pm \right\}$ such that 
\begin{itemize}
\item $k+ |F_1| + |F_2| + \|F_1 \|Ê+ \|F_2 \|$ even
\item $\|ÊF_{2a-1}\|Ê+ \|F_{2a}\| + |F_{2a-1}|Ê+ |F_{2a}|$ even, $\forall \, 2 \leqslant a \leqslant n$,
\end{itemize}
are highest weight vectors, in $\mathcal{M}_k$ of weight $(k)'_+$. In particular, $g_{2k}Ê\, \prod_{s=1}^m L_s^+$ and $g_{2k+1}\, L_1^+\,L_2^- \, \prod_{s=3}^m L_s^+$ are highest weight vectors of weight $(2k)'_+$ resp. $(2k+1)'_+$.  
\end{cor}

Note that the choice for the last factor $F_m \in  \left\{ L_m^\pm, M_m^\pm \right\}$ does not change the results. 

\medskip We again count how many highest weight vectors $g_k\,F$, $F = \prod_{s=1}^m F_s$, with weight $(k)'_+$ we find: for each $F_{2a-1}\,F_{2a}$, $1 \leqslant a \leqslant n$, we have $8$ possible combinations, namely 
$$
\left\{ L_{2a-1}^+\, L_{2a}^+, L_{2a-1}^- \, L_{2a}^-, L_{2a-1}^+\,M_{2a}^+, L_{2a-1}^- \,M_{2a}^-, M_{2a-1}^+\,L_{2a}^+, M_{2a-1}^+ \, M_{2a}^+, M_{2a-1}^- \, L_{2a}^-, M_{2a-1}^- \, M_{2a}^- \right\},
$$
and for $F_m$ we have four possible choices $L_m^\pm$, $M_m^\pm$. Combining this, we find $8^n \, 2^2 = 2^{3n+2} = 2^{2m-n}$ isomorphic irreducible representations with highest weight $(k)'_+$, each of which has dimension 
$$
2^n \, \binom{k+m-2}{k}.
$$
Hence the total dimension of all isomorphic irreducible representations is 
$$
2^{4n+2} \, \binom{k+m-2}{k} = \dim_{\mathbb{C}} \, \mathcal{M}_k,
$$
i.e. the dimensional analysis shows that $\mathcal{M}_k$ may be decomposed as $2^{3n+2}$ isomorphic irreducible representations with highest weight $(k)'_+$. 

\begin{defi}
We define the spinorspace $\mathbb{S}_{2n+1}$ as the image under $\gso(m,\mathbb{C})$ of the idempotent $\prod_{s=1}^m L_s^+$:
$$
\mathbb{S}_{2n+1} = \gso(m,\mathbb{C}) \left( \textup{span}_{\mathbb{C}} \left\{ L_1^+ \,L_2^+ \ldots \, L_{2n}^+ \, L_{2n+1}^+ \right\} \right).
$$
\end{defi}
The element $L_1^+ \,L_2^+ \ldots \, L_{2n}^+  \, L_{2n+1}^+$ is a highest weight vector with weight $(0)'_+ = \left( \frac{1}{2}, \ldots, \frac{1}{2}\right)$ and thus generates an irreducible representation with the same weight. 

\begin{example}
Let $m = 5$ (i.e. $n=2$) and consider $L = L_1^+ \,L_2^+ \, L_3^+ \, L_4^+ \, L_5^+$. We denote $dR(e_{a,b})$ in short as $(a,b)$. The Lie algebra $\gso(5,\mathbb{C})$ is given in this context by the span over $\mathbb{C}$ of the ten elements $(1,2)$, $(1,3)$, $(1,4)$, $(1,5)$, $(2,3)$, $(2,4)$, $(2,5)$, $(3,4)$, $(3,5)$ and $(4,5)$. 
The idempotents involved interact in the following way under the action of $\gso(m,\mathbb{C})$:
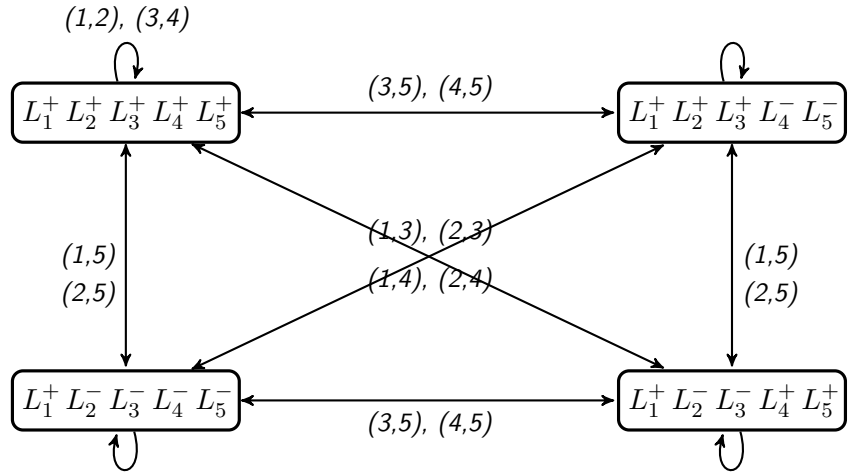
\begin{center}
\begin{tikzpicture}[->,>=stealth',shorten >=1pt,auto,node distance=3cm,
  thick,main node/.style={rectangle,
           rounded corners,
           draw=black, very thick,
           text width=7em,
           minimum height=2em,
           text centered}]

\begin{scope}[]

\node[main node] (Linksboven) {$L_1^+\,L_2^+\,L_3^+ \, L_4^+\,L_5^+$};
\node[main node,below=3cm of Linksboven] (Linksonder) {$L_1^+\,L_2^- \,L_3^- \, L_4^-\,L_5^-$};
\node[main node, right=5cm of Linksboven] (Rechtsboven) {$L_1^+\,L_2^+\,L_3^+ \, L_4^-\,L_5^-$};
\node[main node, right=5cm of Linksonder] (Rechtsonder) {$L_1^+\,L_2^-\,L_3^- \, L_4^+\,L_5^+$};
\path[every node/.style={font=\sffamily\small}]
    (Linksboven) 
       edge[<->]	node[left] (dummyL) {(1,5)} 
    		      	node[below=.5cm of dummyL,left] {(2,5)}
		 	(Linksonder)
       edge [<->] node[above] {(3,5), (4,5)} (Rechtsboven)
       edge [<->] node[above] {(1,3), (2,3)} 
       		        node[below]{(1,4), (2,4)} (Rechtsonder)
       edge [loop above] node {(1,2), (3,4)} (Linksboven)
    (Linksonder) 
       edge[<->] node[below] {(3,5), (4,5)} (Rechtsonder)
       edge[<->] (Rechtsboven)
       edge [loop below] node {} (Linksonder)
    (Rechtsboven)
      edge[<->]	node[right] (dummyR) {(1,5)} 
    		      	node[below=.5cm of dummyR,right] {(2,5)} 
		 	(Rechtsonder)
       edge [loop above] node {} (Rechtsboven)
    (Rechtsonder)
       edge [loop below] node {} (Rechtsonder)      
      ;
\end{scope}
\end{tikzpicture}
\end{center}
Hence
$$
\mathbb{S}_5 = \textup{span}_{\mathbb{C}} \left\{ L_1^+ \,L_2^+ \, L_3^+ \, L_4^+ \, L_5^+, L_1^+ \,L_2^- \, L_3^- \, L_4^+ \, L_5^+, L_1^+ \,L_2^- \, L_3^- \, L_4^- \, L_5^-, L_1^+ \,L_2^+ \, L_3^+ \, L_4^- \, L_5^- \right\}.
$$
The spinorspace $\mathbb{S}_{5}$ is $2^{2}$-dimensional. 
\end{example}

In general, the rotations $dR(e_{2a-1,2a})$, $a= 1,\ldots,n$, acting on an idempotent return the same idempotent up to a multiplicative complex factor. Again, we find that $dR(e_{a,b})$ with $1 \leqslant a, b \leqslant n$, changes the sign of an even number of $L_i$'s. The additional rotations $dR(e_{2a-1,m})$ and $dR(e_{2a,m})$, with $1 \leqslant a \leqslant n$, act as follows on $L$:
$$
L_1^+ \, L_2^+ \ldots \, L_{2n}^+ \, L_{2n+1}^+ \; \mapsto \; L_1^+ \, L_2^+ \ldots L_{2a-1}^+ \, L_{2a}^- \, \ldots \, L_{2n}^- \, L_{2n+1}^-.
$$
The rotation always leaves $L_1^+$ invariant. The resulting idempotent will always have an even number of minus-signs. Starting from the idempotent $L$ with all plus-signs, we thus get all possible idempotents of the following form:
$$
L_1^+ \underbrace{ \, . \quad . \, } \; \underbrace{ \, . \quad . \, } \; \ldots \; \underbrace{ \, . \quad . \, }
$$
where each place $\; \underbrace{ \, . \quad . \, } \; $ consists of either $L_{2a}^+ \, L_{2a+1}^+$ or $L_{2a}^- \, L_{2a+1}^-$, $1 \leqslant a \leqslant n$. We thus get $2^{n}$ spinors and we have the following weight space decomposition
$$
\mathbb{S}_{2n+1} = \bigoplus V_{\left( \pm \frac{1}{2},  \pm \frac{1}{2},\ldots, \pm \frac{1}{2}\right)},
$$
where the sum goes over all weights with an even number of minus-signs.

\begin{example}
Let $m = 7$ (i.e. $n=3$) and consider $L = L_1^+ \,L_2^+ \, L_3^+ \, L_4^+ \, L_5^+\, L_6^+\,L_7^+$.  We will again denote $dR(e_{a,b})$ in short as $(a,b)$. The Lie algebra $\gso(5,\mathbb{C})$ is given in this context by $21$ elements and the corresponding spinorspace will be $8$-dimensional. The idempotents involved interact in the following way under the action of $\gso(m,\mathbb{C})$:
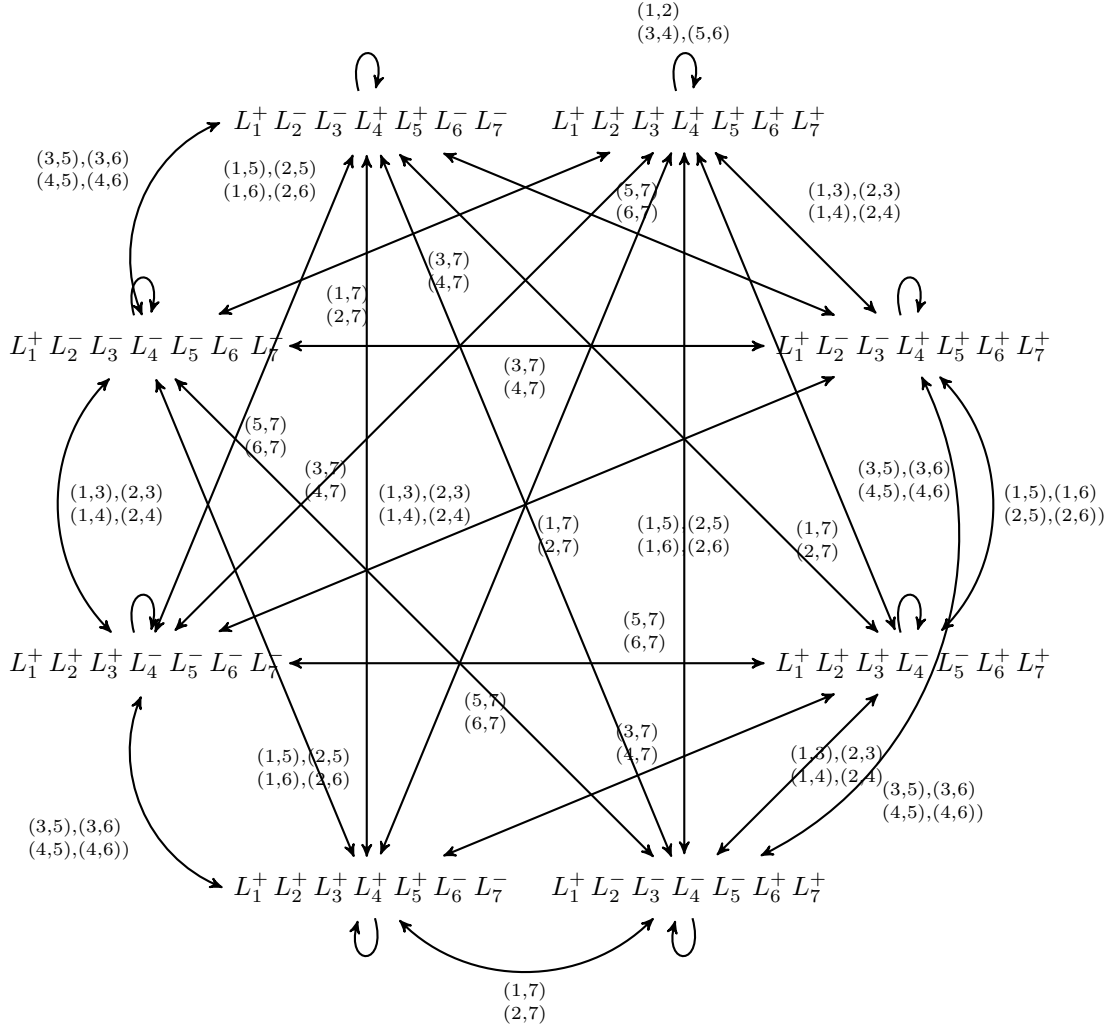
\begin{center}
\begin{tikzpicture}[->,>=stealth',shorten >=1pt,auto,node distance=4cm,
  thick,main node/.style={
           text width=9em,
           minimum height=2em,
           font=\small,
           text centered}]

\begin{scope}[scale=.75]
\node[draw=none,minimum size=11cm,regular polygon,regular polygon sides=8] (a) {};

\node[main node] (P1) at (a.corner 1) {$L_1^+\, L_2^+ \, L_3^+\, L_4^+\, L_5^+ \,L_6^+\, L_7^+$};
\node[main node] (P2) at (a.corner 8) {$L_1^+\, L_2^- \, L_3^- \, L_4^+\, L_5^+ \, L_6^+\, L_7^+$};
\node[main node] (P3) at (a.corner 7) {$L_1^+\, L_2^+ \, L_3^+\, L_4^-\, L_5^- \, L_6^+\, L_7^+$};
\node[main node] (P4) at (a.corner 6) {$L_1^+\, L_2^- \, L_3^-\, L_4^-\, L_5^- \, L_6^+\, L_7^+$};
\node[main node] (P5) at (a.corner 5) {$L_1^+\, L_2^+ \, L_3^+\, L_4^+\, L_5^+ \, L_6^-\, L_7^-$};
\node[main node] (P6) at (a.corner 4) {$L_1^+\, L_2^+ \, L_3^+\, L_4^- \, L_5^- \,L_6^- \,  L_7^-$};
\node[main node] (P7) at (a.corner 3) {$L_1^+\, L_2^- \, L_3^-\, L_4^- \, L_5^- \, L_6^- \, L_7^-$};
\node[main node] (P8) at (a.corner 2) {$L_1^+\, L_2^- \, L_3^-\, L_4^+ \, L_5^+ \, L_6^- \, L_7^-$};

\path[every node/.style={font=\sffamily\small}]
(P1) 
	edge[<->,loop above] node{$\mathop{}^{(1,2)}_{(3,4),(5,6)}$} (P1)
	edge[<->] node{$\mathop{}^{(1,3), (2,3)}_{(1,4), (2,4)}$} (P2)
	edge[<->,near end] node{$\mathop{}^{(3,5), (3,6)}_{(4,5), (4,6)}$} (P3)
	edge[<->,below] node{$\mathop{}^{(1,5), (2,5)}_{(1,6), (2,6)}$} (P4)
	edge[<->,near end] node{$\mathop{}^{(5,7)}_{(6,7)}$} (P5)
	edge[<->,near end,above right] node{$\mathop{}^{(3,7)}_{(4,7)}$} (P6)
	edge[<->,near end] node{$\mathop{}^{(1,7)}_{(2,7)}$} (P7)
(P2) 
	edge[<->,bend left=45] node{$\mathop{}^{(1,5), (1,6)}_{(2,5), (2,6))}$} (P3)
	edge[<->,bend left=45,near end] node{$\mathop{}^{(3,5), (3,6)}_{(4,5), (4,6))}$} (P4)
	edge[<->] node{$\mathop{}^{(1,7)}_{(2,7)}$} (P6)
	edge[<->] node{$\mathop{}^{(3,7)}_{(4,7)}$} (P7)
	edge[<->,midway,above] node{$\mathop{}^{(5,7)}_{(6,7)}$} (P8)
	edge[loop above] node{} (P2)
(P3) 
	edge[loop above] node{} (P3)
	edge[<->,left,near start,below] node{$\mathop{}^{(1,3), (2,3)}_{(1,4), (2,4)}$} (P4)
	edge[<->,above] node{$\mathop{}^{(3,7)}_{(4,7)}$} (P5)
	edge[<->,near start,above] node{$\mathop{}^{(5,7)}_{(6,7)}$} (P6)
	edge[<->,very near start,above] node{$\mathop{}^{(1,7)}_{(2,7)}$} (P8)
(P4) 
	edge[loop below] node{} (P4)
	edge[<->,bend left=45] node{$\mathop{}^{(1,7)}_{(2,7)}$} (P5)
	edge[<->,very near end,right] node{$\mathop{}^{(5,7)}_{(6,7)}$} (P7)
	edge[<->,very near end,below right] node{$\mathop{}^{(3,7)}_{(4,7)}$} (P8)
(P5)  
	edge[loop below] node{} (P5)
	edge[<->,bend left=45] node{$\mathop{}^{(3,5), (3,6)}_{(4,5), (4,6))}$} (P6)
	edge[<->,below,near start] node{$\mathop{}^{(1,5), (2,5)}_{(1,6), (2,6)}$} (P7)
	edge[<->,right] node{$\mathop{}^{(1,3), (2,3)}_{(1,4), (2,4)}$} (P8)
(P6) 
	edge[loop above] node{} (P6)
	edge[<->,bend left=45,right] node{$\mathop{}^{(1,3), (2,3)}_{(1,4), (2,4)}$} (P7)
	edge[<->,very near end] node{$\mathop{}^{(1,5), (2,5)}_{(1,6), (2,6)}$} (P8)
(P7) 
	edge[loop above] node{} (P7)
	edge[<->,bend left=45] node{$\mathop{}^{(3,5), (3,6)}_{(4,5), (4,6)}$} (P8)
(P8) 
	edge[loop above] node{} (P8)
;
\end{scope}
\end{tikzpicture}
\end{center}
Hence
\begin{align*}
\mathbb{S}_7 = \textup{span}_{\mathbb{C}} \left\{ \right. & L_1^+ \,L_2^+ \, L_3^+ \, L_4^+ \, L_5^+ \, L_6^+ \, L_7^+, \;
L_1^+\, L_2^- \, L_3^- \, L_4^+\, L_5^+ \, L_6^+\, L_7^+, \; 
L_1^+\, L_2^+ \, L_3^+\, L_4^-\, L_5^- \, L_6^+\, L_7^+, \\	
&
L_1^+\, L_2^- \, L_3^-\, L_4^-\, L_5^- \, L_6^+\, L_7^+, \; 
L_1^+\, L_2^+ \, L_3^+\, L_4^+\, L_5^+ \, L_6^-\, L_7^-, \; 
L_1^+\, L_2^+ \, L_3^+\, L_4^- \, L_5^- \,L_6^- \,  L_7^-, \\
& 
L_1^+\, L_2^- \, L_3^-\, L_4^- \, L_5^- \, L_6^- \, L_7^-, \; 
L_1^+\, L_2^- \, L_3^-\, L_4^+ \, L_5^+ \, L_6^- \, L_7^- \left. \right\}.
\end{align*}
We indeed find an $8$-dimensional spinorspace $\mathbb{S}_{7}$. 
\end{example}

\section{Conclusion and future research}
The space $\mathcal{M}_k$ of discrete $k$-homogeneous monogenic polynomials is a reducible representation of $\gso(m,\mathbb{C})$, which can, in the odd-dimensional case $m = 2n+1$, be decomposed into $2^{2m-n}$ isomorphic copies of the irreducible $\gso(m,\mathbb{C})$-representation with highest weight $\left( k+\frac{1}{2}, \frac{1}{2},\ldots,\frac{1}{2} \right)$ and in the even-dimensional setting $m =2n$, we find $2^{2m-n}$ isomorphic irreducible representations with highest weight $\left( k+\frac{1}{2},\frac{1}{2},\ldots, \frac{1}{2}\right)$ combined with $2^{2m-n}$ irreps of highest weight $\left( k+ \frac{1}{2}, \frac{1}{2},\ldots, \frac{1}{2},-\frac{1}{2}\right)$. This is done by means of an appropriate amount of idempotents. 

\medskip Let $g_k =  \left( \xi_2 - \xi_1 \right) \left( \xi_2 + \xi_1 \right)\left( \xi_2 - \xi_1 \right) \left( \xi_2 + \xi_1 \right) \ldots[1]$, $(k \text{ factors})$, be a discrete homogeneous monogenic function of degree $k$ and let 
\begin{align*}
L^\pm_{2a-1} &= \left( \e_{2a-1}^+ \e_{2a-1}^- \pm i\, \e_{2a-1}^+ \right), & 
L^\pm_{2a} &= \left( \e_{2a}^+ \e_{2a}^- \pm \e_{2a}^+ \right), \\
M^\pm_{2a-1} &= \left( \e_{2a-1}^- \e_{2a-1}^+ \pm i\, \e_{2a-1}^- \right), &
M^\pm_{2a} &= \left( \e_{2a}^- \e_{2a}^+ \pm \e_{2a}^- \right),
\end{align*}

Denote $\| L^\pm_a\| = 0$, $\|M^\pm_aÊ\| = 1$ and $| L^+_a | = |M_a^- | = 0$ and $|L^-_a| = |M_a^+| = 1$.

\medskip In even dimension $m = 2n$, the polynomial $g_k \, F \in \mathcal{M}_k$, $F = \prod_{s=1}^m F_s$ with $F_s \in \left\{ L_s^\pm, M_s^\pm \right\}$, is a weight vector of $\gso(m,\mathbb{C})$ with 
\begin{itemize}
\item weight $(k)'_+$ when $k+ |F_1| + |F_2| + \|F_1 \|Ê+ \|F_2 \|$ is even and $\|ÊF_{2a-1}\|Ê+ \|F_{2a}\| + |F_{2a-1}|Ê+ |F_{2a}|$ is even for $2 \leqslant a \leqslant n$.

\item weight $(k)'_-$ when $k+ |F_1| + |F_2| + \|F_1 \|Ê+ \|F_2 \|$ is even, $\|ÊF_{2a-1}\|Ê+ \|F_{2a}\| + |F_{2a-1}|Ê+ |F_{2a}|$ is even for $2 \leqslant a \leqslant n-1$ and $\|ÊF_{2n-1}\|Ê+ \|F_{2n}\| + |F_{2n-1}|Ê+ |F_{2n}|$ is odd.\end{itemize}
We find $2^{2m-n}$ highest weight vectors in $\mathcal{M}_{k}$ with weight $\left(kÊ\right)'_+$ and $2^{2m-n}$ weight vectors, with weight $(k)'_-$.

\medskip In odd dimensions $m=2n+1$, the polynomial $g_k \, F \in \mathcal{M}_k$, $F = \prod_{s=1}^m F_s$ with $F_s \in \left\{ L_s^\pm, M_s^\pm \right\}$, is a weight vector of $\gso(m,\mathbb{C})$ with weight $(k)'_+$ when $k+ |F_1| + |F_2| + \|F_1 \|Ê+ \|F_2 \|$ is even and $\|ÊF_{2a-1}\|Ê+ \|F_{2a}\| + |F_{2a-1}|Ê+ |F_{2a}|$ is even for $2 \leqslant a \leqslant n$. We find $2^{2m-n}$ highest weight vectors in $\mathcal{M}_{k}$ with weight $\left(kÊ\right)'_+$.

\medskip We have proven throughout this article how the spaces $\mathcal{H}_k$ and $\mathcal{M}_k$ of harmonic resp. monogenic discrete $k$-homogeneous polynomials may be decomposed into irreducible representations of $\gso(m,\mathbb{C})$. However, because of the presence of the basiselements $e_a$ and $e_a^\perp$ in the definition of the generators of the rotations, the spinorspace is no maximal left ideal. In future research we will investigate other possibilities to define rotations and the spinorspace in the hopes of writing the spinorspace as maximal left ideal. An equivalent description of $\mathcal{H}_k$ and $\mathcal{M}_k$ as $SO(m)$-representations is also still work in progress.

\section{Acknowledgments}
The first author acknowledges the support of the Research Foundation - Flanders (FWO), grant no. FWO13$\backslash$PDO$\backslash$039.

\end{document}